\documentclass[12pt,reqno]{amsart}
\usepackage{amsmath,amssymb,extarrows}
\usepackage{url}
\usepackage{tikz,enumerate}
\usepackage{diagbox}
\usepackage{appendix}
\usepackage{epic}
\usepackage{comment}

\usepackage{longtable} 

\usepackage{float} 

\usepackage{cite}

\usepackage{hyperref}
\usepackage{array}

\usepackage{booktabs}

\allowdisplaybreaks[4]

\usepackage{makecell}
\usepackage{float} 
\usepackage{siunitx} 
\usepackage{multirow}
\usepackage[figuresright]{rotating}

\newtheorem{theorem}{Theorem}

\newtheorem{conj}[theorem]{Conjecture}

\theoremstyle{definition}

\theoremstyle{remark}

\numberwithin{equation}{section}
\numberwithin{theorem}{section}
\numberwithin{defn}{section}

\begin{document}
\title[Modular Nahm Sums for the Inverse Cartan Matrix of Type $D_r$]{Modular Nahm Sums for the Inverse Cartan Matrix of Type $D_r$}

\author{Liuquan Wang and Shangwen Wang}

\address[L.\ Wang]{School of Mathematics and Statistics, Wuhan University, Wuhan 430072, Hubei, People's Republic of China}

\email{wanglq@whu.edu.cn;mathlqwang@163.com}

\address[S.\ Wang]{School of Mathematics and Statistics, Wuhan University, Wuhan 430072, Hubei, People's Republic of China}
\email{whuwsw@whu.edu.cn}



\begin{abstract}
For $r\geq 3$ we denote by $\mathcal{C}(D_r)$ the Cartan matrix of type $D_r$. Recently, Sun and Wang conjectured a Rogers--Ramanujan type identity for the Nahm sum associated with $\mathcal{C}(D_r)^{-1}$ and the zero vector. They further conjecture that there exist $r-1$ companion modular Nahm sums associated with nonzero vectors. We partially prove this conjecture by constructing $\lfloor (r+4)/2\rfloor$ modular Nahm sums for $\mathcal{C}(D_r)^{-1}$. To prove their modularity, we utilize the method of Bailey pairs to establish various Rogers--Ramanujan type identities. In particular, we confirm their conjectural identity.
\end{abstract}

\maketitle

\section{Introduction}\label{sec-intro}

The famous Rogers--Ramanujan identities \cite{Rogers} express two  $q$-hypergeometric series as some elegant infinite products:
\begin{align}\label{RR}
    \sum_{n=0}^\infty \frac{q^{n^2+in}}{(q;q)_n}=\frac{1}{(q^{i+1},q^{4-i};q^5)_\infty}, \quad i=0,1.
\end{align}
Here and throughout this paper we use $q$-series notations: for $n\in \mathbb{Z}_{\geq 0} \cup \{\infty\}$,
\begin{align}
    (a;q)_n:=\prod\limits_{k=0}^{n-1}(1-aq^k),\quad (a_1,a_2,\dots,a_m;q)_n:=\prod\limits_{k=1}^m (a_k;q)_n.
\end{align}
We always assume $q=e^{2\pi i\tau}$ with $\mathrm{Im}~\tau>0$ so that $|q|<1$. After multiplying by $q^{-1/60}$ for $i=0$ and $q^{11/60}$ for $i=1$, the right side of \eqref{RR} becomes a modular function (of weight zero). This reveals the modularity of the series in the left side. A natural question then arises: what kind of $q$-hypergeometric series are modular? 

An important special case of the above question is Nahm's problem \cite{Nahm1994,Nahm2007}. Let $r$ be a positive integer, $A$ a $r\times r$ positive definite matrix, $B$ a column vector of length $r$ and $C$ a scalar. We define the Nahm sum generated by $(A,B,C)$ as
\begin{align}
    f_{A,B,C}(q):=\sum_{n=(n_1,n_2,\dots,n_r)\in \mathbb{Z}_{\geq 0}^r} \frac{q^{\frac{1}{2}n^\mathrm{T}An+n^\mathrm{T}B+C}}{(q;q)_{n_1}\cdots (q;q)_{n_r}}.
\end{align}
Nahm's problem is to find all such $(A,B,C)$ with rational entries so that $f_{A,B,C}(q)$ is modular. Modular Nahm sums are connected to  rational two-dimensional conformal field theories  (2d CFTs), and have also played important roles in combinatorics and the representation theory of Lie algebras. 

Around 2007, Zagier \cite{Zagier} solved Nahm's problem when the rank $r=1$ by showing that there are exactly seven rank one modular Nahm sums. Zagier also provided many candidates of modular Nahm sums in the rank two and rank three cases. The modularity of these candidates have been confirmed by works of Zagier \cite{Zagier}, Vlasenko--Zwegers \cite{VZ}, Cherednik--Feigin \cite{Feigin}, Wang \cite{Wang-rank2,Wang-rank3} and Cao--Rosengren--Wang \cite{CRW}. Recent works of Cao and Wang \cite{Cao-Wang2024,Cao-Wang2025} provided more rank three and rank four modular Nahm sums. So far, a complete classification of modular Nahm sums seems out of reach.

As a rich source of examples for Nahm's problem, Nahm sums associated with Cartan matrices of Dynkin diagrams are usually modular. For any Dynkin diagram $X$, we denote by $\mathcal{C}(X), r(X)$ and $h(X)$ its Cartan matrix, rank and Coxeter number, respectively.  Inspired by the study of fermionic
sum representations for characters of coset 2d CFTs (see e.g., \cite{Kedem,Kedem-Klassen,Keegan}), the following folklore conjecture (see e.g.,~\cite{Keegan,Lee,SW})  reveals a way to construct modular Nahm sums using Cartan matrices of Dynkin diagrams.
\begin{conj}\label{conj-folklore}
Let $X$ and $Y$ be any Dynkin diagrams of ADET type. The Nahm sums associated with $(M(X,Y),0,C(X,Y))$ is modular where 
\begin{align}
M(X,Y)=\mathcal{C}(X)\otimes \mathcal{C}(Y)^{-1}, \quad C(X,Y)=-\frac{1}{24}\frac{r(X)r(Y)h(X)}{h(X)+h(Y)}.    
\end{align}    
Here $\otimes$ denotes the Kronecker product of matrices.
\end{conj}
From some known Rogers--Ramanujan type identities, this conjecture has been justified for several infinite families. For instance, the Andrews--Gordon identities \cite{Andrews1974} assert that for integers $r,s$ such that $r\geq 1$ and $1\leq s \leq r+1$,
\begin{align}
\sum_{n_1,\dots,n_{r}\geq 0} \frac{q^{N_1^2+\cdots+N_{r}^2+N_s+\cdots +N_{r}}}{(q;q)_{n_1}\cdots (q;q)_{n_{r}}}  =\frac{(q^s,q^{2r+3-s},q^{2r+3};q^{2r+3})_\infty}{(q;q)_\infty} \label{AG}
\end{align}
where $N_j=n_j+\cdots+n_{r}$ ($1\leq j\leq r$) and $N_{r+1}=0$.
The Nahm sum here is associated with the matrix $\mathcal{C}(A_1)\otimes \mathcal{C}(T_r)^{-1}$. In particular, the $s=r+1$ case confirms the above conjecture for $(A_1,T_r)$.
More generally,  the Nahm sum associated with $(A_{2k-1}, T_r)$ is modular for any $k\geq 1$ and $r\geq 1$. This follows from a recent work of  Creutzig and Garner  \cite[Theorem 1.0.1]{CG} which confirms a conjecture of Warnaar and Zudilin \cite[Conjecture 1.1 and Theorem 1.2]{WZ}.

Besides, the conjecture has also been proved for the case $(T_1,T_r)$. The modularity of the Nahm sum associated with $(T_1,T_{r})$ follows from the following identity proved by Stembridge \cite[Corollary 1.5(b)]{Stembridge}:
\begin{align}\label{eq-Stembridge}
    \sum_{n_1,\dots,n_r\geq 0} \frac{q^{\frac{1}{2}(N_1^2+N_2^2+\cdots+N_r^2)}}{(q;q)_{n_1}\cdots (q;q)_{n_r}}=\frac{(-q^{1/2};q)_\infty(q^{(r+1)/2},q^{(r+3)/2},q^{r+2};q^{r+2})_\infty}{(q;q)_\infty}.
\end{align}
When the rank $r$ is even, this identity was proposed as a conjecture by Melzer \cite{Melzer}. Different proofs for the even and odd rank cases were later provided by Bressoud--Ismail--Stanton \cite{Bressoud2000} and Warnaar \cite{Warnaar-2003}, respectively.

Conjecture \ref{conj-folklore} also holds for $(A_1,D_r)$ for any $r\geq 3$. In this case, the matrix part is $\mathcal{C}(A_1)\otimes \mathcal{C}(D_r)=2\mathcal{C}(D_r)^{-1}$. Around 1993,  Kedem, Klassen, McCoy and Melzer \cite[(2.9)]{Kedem} conjectured that for $a=0,1$,
\begin{align}\label{eq-KKMM}
\sum_{\substack{n=(n_1,\dots,n_r)\in \mathbb{N}^r \\ n_{r-1}+n_r\equiv a \!\!\! \pmod{2}}}  \frac{q^{n^\mathrm{T}\mathcal{C}(D_r)^{-1}n}}{(q;q)_{n_1}\cdots (q;q)_{n_r}}=\frac{1}{(q;q)_\infty}\sum_{n \in \mathbb{Z}} q^{r(n+\frac{1}{2}a)^2}.
\end{align}
Flohr, Grabow and Koehn \cite[Eqs.\ (2.28)--(2.30)]{FGK} conjectured three families of similar identities which include \eqref{eq-KKMM} as special cases. All of their conjectural identities were confirmed by Warnaar \cite{Warnaar}, who also proved an additional family of companion identities.  Recently, the authors \cite{Wang-Wang} provided new proofs for these identities and extended \eqref{eq-KKMM} to a general identity which contains two free parameters. 

As for the case $(T_r,T_1)$, Conjecture \ref{conj-folklore} without the formula for $C(T_r,T_1)$ was proposed as a conjecture by Calinescu--Milas--Penn \cite{CMP} when studying twisted modules of principal
subspace vertex algebras. They confirmed the case $r=2$ which coincides with one of Zagier's rank two examples \cite{Zagier}. The Nahm sum associated with $M(T_r,T_1)=\mathcal{C}(T_r)$ is usually called tadpole Nahm sums since it is related to the tadpole Dynkin diagram. The modularity for the cases $r=3$ and $r=4,5$ were confirmed by Milas--Wang \cite{MW24} and Shi--Wang \cite{Shi-Wang}, respectively.  

The cases $(D_r,T_1)$ and $(T_1,D_r)$  ($r\geq 3$) were confirmed only in the cases $r=3$ by Cao and Wang \cite{Cao-Wang2024}. In the case $(D_r,T_1)$, the matrix is $\mathcal{C}(D_r)=(d_{ij})_{r\times r}$ with entries
\begin{equation}
\begin{split}
   & d_{ii}=2, \quad 1\leq i \leq r, \quad d_{i,j}=-1, \quad |i-j|=1,~~ 1\leq i,j\leq r-1, \\
   & d_{r-2,r}=d_{r,r-2}=-1, \quad \text{and} \quad d_{i,j}=0 \quad \text{otherwise}.
\end{split}
\end{equation}
In the case $(T_1,D_r)$, the matrix is $\mathcal{C}(D_r)^{-1}=(a_{ij})_{r\times r}$ with entries
\begin{equation}\label{eq-2D-inverse}
\begin{split}
&a_{r,r-1}=a_{r-1,r}=\tfrac{1}{2}r-1, \quad a_{r-1,r-1}=a_{r,r}=\tfrac{1}{2}r, \\
&a_{i,r-1}=a_{i,r}=a_{r-1,i}=a_{r,i}=i, \quad 1\leq i\leq r-2, \\
&a_{ij}=2\min(i,j), \quad 1\leq i,j\leq r-2.
\end{split}
\end{equation}
Recently, Sun and Wang \cite[Conjecture 1.1]{SW} extended Conjecture \ref{conj-folklore} to Dynkin diagrams of type ABCDEFGT in the context of generalized Nahm sums. Since we will not discuss generalized Nahm sums, we refer the reader to \cite{SW} for details. Sun and Wang also observed some interesting relations between Nahm sums and 2d CFTs. In particular, they considered Nahm sums associated with $(T_1, D_r)$ ($r\geq 3$). In this case, they find that the theory is the effective $\mathcal{N} = 1$ supersymmetric Virasoro minimal model ${\rm SM}_{\mathrm{eff}}(8r + 4, 2)$. For Nahm sums associated with the zero vector, they \cite[(1.12)]{SW} conjectured the following identity revealing relation between the Nahm sum and the fermionic characters:
\begin{align}\label{eq-SW-conj}
f_{\mathcal{C}(D_r)^{-1},0,0}(q)=&\prod_{\substack{n=1\\ n\not\equiv 2 \bmod 4\\ n\not\equiv 0,\, \pm (4r+1) \bmod (8r+4)}}^\infty \big(1-q^{n/2}\big)^{-1} + q^{\frac{r}{2}} \cdot \prod_{\substack{n=1\\ n\not\equiv 2 \bmod 4\\ n\not\equiv 0,\, \pm 1 \bmod (8r+4)}}^\infty \big(1-q^{n/2}\big)^{-1} \nonumber \\
&+ 2q^{\frac{r}{8}}\cdot \prod_{\substack{n=1\\ \text{$n$ odd}}}^\infty \big(1-q^{n}\big)^{-1}  \prod_{\substack{n=1\\ n\not\equiv 0,\, \pm (r+1) \bmod (4r+2)}}^\infty \big(1-q^{n}\big)^{-1}.
\end{align}
They further conjecture that there exist $r-1$ non-zero vectors $B$ so that the associated Nahm sums are modular and have similar expressions in terms of NS and R characters of $\mathrm{SM}_{\rm eff}(8r+4,2)$.

In this paper, we will provide a proof of the conjectural identity \eqref{eq-SW-conj} and thus confirm Conjecture \ref{conj-folklore} for the case $(T_1,D_r)$. Moreover, we make some progress towards the conjecture of Sun and Wang \cite{SW} regarding existence of companion modular Nahm sums. We prove that for each $r\geq 3$, there are $\lfloor (r+4)/2\rfloor$ vectors $B$ such that $f_{\mathcal{C}(D_r)^{-1},B,C}(q)$ is modular for suitable $C$. To state our results in a compact way, for $0<a < m$ we denote
\begin{align}\label{J-defn}
    J_m:=(q^m;q^m)_\infty, \quad J_{a,m}:=(q^a,q^{m-a},q^m;q^m)_\infty.
\end{align}
\begin{theorem}\label{thm-main-1}
For $\lambda \in \{0,1,2,\dots,\lfloor r/2\rfloor\}$, we have
\begin{align}\label{main-1}
 f_{\mathcal{C}(D_{r})^{-1},B_{\lambda},0}(q)&=\frac{(-q^{1/2};q)_{\infty}}{(q;q)_{\infty}} J_{2r+2\lambda+1/2,4r+2}
    -q^{(r+1)/2-\lambda}\frac{(-q^{1/2};q)_{\infty}}{(q;q)_{\infty}}J_{2\lambda-1/2,4r+2} \nonumber \\
    &\qquad+2q^{\frac{r+4\lambda}{8}}\frac{(-q;q)_{\infty}}{(q;q)_{\infty}} J_{r+1-2\lambda,4r+2}
\end{align}
where the vector $B_\lambda \in \mathbb{Q}^r$ is given by
\begin{align}
     n^{\mathrm{T}}B_{\lambda}= \sum_{\ell=r-2\lambda +1}^{r-2}\Big\lfloor \frac{\ell +2\lambda +1-r}{2} \Big\rfloor n_{\ell}+\frac{\lambda}{2}(n_{r-1}+n_{r}). 
\end{align}
As a consequence,  $f_{\mathcal{C}(D_r)^{-1},B_\lambda,C_\lambda}(q)$ is modular with $C_\lambda=\frac{8\lambda^2-4\lambda-r}{8(2r+1)}$.
\end{theorem}
In particular, when $\lambda=0$, we confirm the conjectural identity \eqref{eq-SW-conj}. Since $r(T_1)=1$, $r(D_r)=r$, $h(T_1)=3$ and $h(D_r)=2r-2$, Conjecture \ref{conj-folklore} predicates that we need $C=C(T_1,D_r)=-\frac{r}{8(2r+1)}$ so that $f_{\mathcal{C}(D_r)^{-1},0,C}(q)$ is modular. The value of $C$ is exactly $C_0$. Hence the $\lambda=0$ case of Theorem \ref{thm-main-1} also confirms Conjecture \ref{conj-folklore} for the $(T_1,D_r)$ case.

Meanwhile, we find the following companion modular Nahm sums.
\begin{theorem}\label{thm-main-2}
Let
\begin{align}
   B^{(0)}&=\frac{1}{2}(1,0,1,0,\dots,1,0,0,1)^{\mathrm{T}}\in \mathbb{Q}^{2k}, \\
   B^{(1)}&=\frac{1}{4}(-2,0,-2,0,\cdots,-2,0,-2,1,-1)^\mathrm{T} \in \mathbb{Q}^{2k+1}.  
\end{align}
We have
\begin{align}
 f_{\mathcal{C}(D_{2k})^{-1},B^{(0)},0}(q)&=\frac{(-q;q)_\infty}{(q;q)_\infty}J_{4k+1,8k+2} +q^{\frac{1}{4}k}\frac{(-q^{1/2};q)_\infty}{(q;q)_\infty}J_{2k+\frac{1}{2},8k+2}, \label{eq-thm2-even} \\
     f_{\mathcal{C}(D_{2k+1})^{-1},B^{(1)},0}(q)&=2\frac{(-q;q)_\infty}{(q;q)_\infty}J_{4k+2,8k+6} +q^{\frac{2k-1}{8}}\frac{(-q^{{1}/{2}};q)_\infty}{(q;q)_\infty}J_{2k+\frac{5}{2},8k+6}  \nonumber \\
     &\qquad +q^{\frac{2k+3}{8}}\frac{(-q^{{1}/{2}};q)_\infty}{(q;q)_\infty}J_{2k+\frac{1}{2},8k+6}. \label{eq-thm2-odd}
\end{align}
As a consequence, $f_{\mathcal{C}(D_{2k+i})^{-1},B^{(i)},C^{(i)}} (q)$ ($i=0,1$) is modular with $C^{(0)}=0$ and $C^{(1)}=\frac{1}{4(4k+3)}$.
\end{theorem}

Theorems \ref{thm-main-1} and \ref{thm-main-2} provide $\lfloor (r+4)/2\rfloor$ different modular Nahm sums for the matrix $\mathcal{C}(D_r)^{-1}$. It remains open to find $\lfloor (r-3)/2\rfloor$ other modular Nahm sums for $\mathcal{C}(D_r)^{-1}$ to fully prove the aforementioned conjecture of Sun and Wang.  We are not sure whether the results in Theorem \ref{thm-main-2} can be extended to a general family of modular Nahm sums similar to Theorem \ref{thm-main-1}.  

The rest of this paper is organized as follows. In Section \ref{sec-pre} we collect some auxiliary identities and review some knowledge from the theory of Bailey pairs. We present proofs for Theorems \ref{thm-main-1} and \ref{thm-main-2} in Section \ref{sec-proof}.

\section{Preliminaries}\label{sec-pre}
Recall the Jacobi triple product identity \cite[Theorem 2.8]{Andrews1998}
\begin{align}\label{jtp}
    \sum_{n=-\infty}^\infty x^nq^{n^2}=(-qx,-q/x,q^2;q^2)_\infty, \quad x\neq 0.
\end{align}
As some finite versions of it, we have the following identities  \cite[Eqs.~(4.3) and (4.10)]{Andrews2016}:
\begin{align}
    \sum_{i=-n}^n \frac{(-1)^iq^{i(i-1)/2}x^{-i}}{(q;q)_{n-i}(q;q)_{n+i}}&=\frac{(x^{-1};q)_n(xq;q)_n}{(q;q)_{2n}}, \label{Andrews1} \\
    \sum_{i=-n-1}^n \frac{(-1)^iq^{i(i-1)/2}x^{-i}}{(q;q)_{n-i}(q;q)_{n+i+1}}&=\frac{(x^{-1};q)_n(xq;q)_{n+1}}{(q;q)_{2n+1}}. \label{Andrews2}
\end{align}

A pair of sequences $(\alpha_n(a;q),\beta_n(a;q))$ is called a Bailey pair relative to the parameter $a$ if for all $n\geq 0$,
 \begin{align}\label{defn-BP}
     \beta_n(a;q)=\sum_{k=0}^n\frac{\alpha_k(a;q)}{(q;q)_{n-k}(aq;q)_{n+k}}.
 \end{align}
For the remainder of this section, we assume that $(\alpha_n(a;q),\beta_n(a;q))$ is a Bailey pair relative to $a$. As direct consequences of the Bailey lemma (see, e.g.,~ \cite[Theorem 2.1]{Warnaar-2001}), we obtain a Bailey pair $(\alpha_n',\beta_n')$ from the Bailey pair $(\alpha_n,\beta_n)$ through the following formulas (see e.g.,~\cite[(S1)(S2)(S3)(S4)(S5)(S6)]{Bressoud2000}):
\begin{equation}\label{BP-S1}
\begin{split}
\alpha_n'(a;q)&=a^nq^{n^2}\alpha_n(a;q), \\
\beta_n'(a;q)&=\sum_{k=0}^n \frac{a^kq^{k^2}}{(q;q)_{n-k}}\beta_k(a;q);
\end{split}\tag{S1}
\end{equation}
\begin{equation}\label{S2}
\begin{split}
    \alpha_n'(a;q)&=a^{n/2}q^{n^2/2}\alpha_n(a;q), \\
    \beta_n'(a;q)&=\sum_{k=0}^n\frac{(-\sqrt{aq};q)_k}{(q;q)_{n-k}(-\sqrt{aq};q)_n}a^{k/2}q^{k^2/2}\beta_k(a;q);
\end{split} \tag{S2}
\end{equation}
\begin{equation}\label{S3}
\begin{split}
\alpha_n'(a;q)&=\frac{(-q^{1/2};q)_n}{(-aq^{1/2};q)_n}a^nq^{n^2/2}\alpha_n(a;q), \\
    \beta_n'(a;q)&=\sum_{k=0}^n \frac{(-q^{1/2};q)_k}{(q;q)_{n-k}(-aq^{1/2};q)_n}a^kq^{k^2/2}\beta_k(a;q);
\end{split}\tag{S3}
\end{equation}
\begin{equation}\label{S4}
    \begin{split}
        \alpha_n'(a;q)&=\frac{(-aq^{1/2};q)_n}{(-q^{1/2};q)_n}q^{n^2/2}\alpha_n(a;q), \\
        \beta_n'(a;q)&=\sum_{k=0}^n\frac{(-aq^{1/2};q)_k}{(q;q)_{n-k}(-q^{1/2};q)_n}q^{k^2/2}\beta_k(a;q);
    \end{split}\tag{S4}
\end{equation}
\begin{equation}\label{S5}
    \begin{split}
        \alpha_n'(a;q)&=\frac{(-a^{1/2}q;q)_n}{(-a^{1/2};q)_n}a^{n/2}q^{(n^2-n)/2}\alpha_n(a;q), \\
        \beta_n'(a;q)&=\sum_{k=0}^n \frac{(-a^{1/2}q;q)_k}{(q;q)_{n-k}(-a^{1/2};q)_n}a^{k/2}q^{(k^2-k)/2}\beta_k(a;q);
    \end{split}\tag{S5}
\end{equation}
\begin{equation}\label{S6}
    \begin{split}
        \alpha_n'(a;q)&=\frac{(-a^{1/2};q)_n}{(-a^{1/2}q;q)_n}a^{n/2}q^{(n^2+n)/2}\alpha_n(a;q), \\
        \beta_n'(a;q)&=\sum_{k=0}^n\frac{(-a^{1/2};q)_k}{(q;q)_{n-k}(-a^{1/2}q;q)_n}a^{k/2}q^{(k^2+k)/2}\beta_k(a;q).
    \end{split}\tag{S6}
\end{equation}

When $n\rightarrow \infty$, we deduce from \eqref{defn-BP} and \eqref{BP-S1} that
\begin{align}\label{eq-BP-id-key}
\sum_{n=0}^\infty a^nq^{n^2}\beta_n(a;q)=\frac{1}{(aq;q)_\infty} \sum_{n=0}^\infty a^n q^{n^2}\alpha_n(a;q).
\end{align}

If we set $a=1$ in \eqref{S2}, then we obtain a Bailey pair
\begin{equation}\label{BP-S2-a1}
\begin{split}
    \alpha_n'(1;q)&=q^{n^2/2}\alpha_n(1;q), \\
    \beta_n'(1;q)&=\sum_{k=0}^n \frac{(-q^{1/2};q)_kq^{k^2/2}}{(q;q)_{n-k}(-q^{1/2};q)_n}\beta_k(1;q).
\end{split}
\end{equation}
Let $n\rightarrow \infty $. We obtain the transformation formula
\begin{align}\label{BP-S2-a1-infinite}
    \sum_{n=0}^\infty (-q^{1/2};q)_nq^{n^2/2}\beta_n(1;q)=\frac{(-q^{1/2};q)_\infty}{(q;q)_\infty}\sum_{n=0}^\infty q^{n^2/2}\alpha_n(1;q).
\end{align}

If we set $a=q$ in \eqref{S2}, then we obtain a Bailey pair
\begin{equation}\label{BP-S2-aq}
\begin{split}
    \alpha_n'(q;q)&=q^{n(n+1)/2}\alpha_n(q;q), \\
    \beta_n'(q;q)&=\sum_{k=0}^n \frac{(-q;q)_kq^{k(k+1)/2}}{(q;q)_{n-k}(-q;q)_n}\beta_k(q;q).
\end{split}
\end{equation}
Let $n\rightarrow \infty $. We obtain the transformation formula
\begin{align}\label{BP-S2-aq-infinite}
    \sum_{n=0}^\infty (-q;q)_nq^{n(n+1)/2}\beta_n(q;q)=\frac{(-q;q)_\infty}{(q^2;q)_\infty}\sum_{n=0}^\infty q^{n(n+1)/2}\alpha_n(q;q).
\end{align}

If we set $a=q$ in \eqref{S3}, then we obtain a Bailey pair
\begin{equation}\label{S3-aq}
\begin{split}
    \alpha_n'(q;q)&=\frac{1+q^{\frac{1}{2}}}{1+q^{n+\frac{1}{2}}}q^{\frac{1}{2}n^2+n}\alpha_n(q;q), \\
    \beta_n'(q;q)&=\sum_{k=0}^n \frac{(-q^{\frac{1}{2}};q)_k}{(q;q)_{n-k}(-q^{\frac{3}{2}};q)_n}q^{\frac{1}{2}k^2+k}\beta_k(q;q).
\end{split}
\end{equation}
Let $n\rightarrow \infty$. We obtain
\begin{align}\label{S3-aq-infinite}
\sum_{n=0}^\infty q^{\frac{1}{2}n^2+n} (-q^{\frac{1}{2}};q)_n\beta_n(q;q)=\frac{(-q^{\frac{1}{2}};q)_\infty}{(q^2;q)_\infty} \sum_{n=0}^\infty \frac{q^{\frac{1}{2}n^2+n}}{1+q^{n+\frac{1}{2}}}\alpha_n(q;q).
\end{align}

If we set $a=q$ in \eqref{S4}, then we obtain a Bailey pair
\begin{equation}\label{S4-aq}
    \begin{split}
\alpha_n'(q;q)&=\frac{1+q^{n+\frac{1}{2}}}{1+q^{\frac{1}{2}}}q^{\frac{1}{2}n^2}\alpha_n(q;q), \\
\beta_n'(q;q)&=\sum_{k=0}^n\frac{(-q^{\frac{3}{2}};q)_k}{(q;q)_{n-k}(-q^{\frac{1}{2}};q)_n}q^{\frac{1}{2}k^2}\beta_k(q;q).
    \end{split}
\end{equation}
Let $n\rightarrow \infty$. We obtain
\begin{align}\label{S4-aq-infinite}
    \sum_{n=0}^\infty q^{\frac{1}{2}n^2}(-q^{\frac{3}{2}};q)_n\beta_n(q;q)=\frac{(-q^{\frac{1}{2}};q)_\infty}{(q^2;q)_\infty} \sum_{n=0}^\infty \frac{1+q^{n+\frac{1}{2}}}{1+q^{\frac{1}{2}}}q^{\frac{1}{2}n^2}\alpha_n(q;q).
\end{align}

If we set $a=1$ in \eqref{S5}, then we obtain a Bailey pair:
\begin{equation}\label{S5-special}
\begin{split}
   \alpha_n'(1;q)&=\frac{(-q;q)_n}{(-1;q)_n}q^{(n^2-n)/2}\alpha_n(1;q), \\
    \beta_n'(1;q)&=\sum_{k=0}^n\frac{(-q;q)_k}{(q;q)_{n-k}(-1;q)_n}q^{(k^2-k)/2}\beta_k(1;q).
\end{split}
\end{equation}
Let $n\rightarrow \infty$. We deduce that
\begin{align}\label{S5-a1-infinite}
\sum_{n=0}^\infty q^{(n^2-n)/2}(-q;q)_n\beta_n(1;q)=\frac{(-1;q)_\infty}{(q;q)_\infty}\sum_{n=0}^\infty \frac{(-q;q)_n}{(-1;q)_n}q^{(n^2-n)/2}\alpha_n(1;q).
\end{align}

If we set $a=1$ in \eqref{S6}, then we obtain a Bailey pair:
\begin{equation}\label{S6-special}
    \begin{split}
        \alpha_n'(1;q)&=\frac{(-1;q)_n}{(-q;q)_n}q^{(n^2+n)/2}\alpha_n(1;q), \\
        \beta_n'(1;q)&=\sum_{k=0}^n\frac{(-1;q)_k}{(q;q)_{n-k}(-q;q)_n}q^{(k^2+k)/2}\beta_k(1;q).
    \end{split}
\end{equation}
Let $n\rightarrow \infty$. We obtain
\begin{align}\label{S6-a1-infinite}
    \sum_{n=0}^\infty q^{(n^2+n)/2}(-1;q)_n\beta_n(1;q)=\frac{(-q;q)_\infty}{(q;q)_\infty}\sum_{n=0}^\infty \frac{(-1;q)_n}{(-q;q)_n}q^{(n^2+n)/2}\alpha_n(1;q).
\end{align}

For our proofs we also need to construct Bailey pairs with different parameters from a given Bailey pair. Suppose that $(\alpha_n(a;q),\beta_n(a;q))$ is a Bailey pair relative to $a$. From  \cite[(2.4) and (2.5)]{Lovejoy2004} we know that $(\alpha_n',\beta_n')$ is a Bailey pair relative to $aq$ where
\begin{equation}\label{eq-BP-lift}
\begin{split}
&\alpha_n'(aq;q)=\frac{(1-aq^{2n+1})a^nq^{n^2}}{1-aq}\sum_{k=0}^n a^{-k}q^{-k^2}\alpha_k(a;q), \\
&\beta_n'(aq;q)=\beta_n(a;q).
\end{split}
\end{equation}
Moreover, from \cite[Theorem 3.2]{Warnaar-2001} we know that $(\widetilde{\alpha}_n,\widetilde{\beta}_n)$ is a Bailey pair relative to $a/q$ where
\begin{align}\label{eq-BP-reduce}
\widetilde{\alpha}_0(a/q;q)&=\alpha_0(a;q), \quad \widetilde{\alpha}_n(a/q;q)=(1-a)\Big(\frac{\alpha_n(a;q)}{1-aq^{2n}}-\frac{aq^{2n-2}\alpha_{n-1}(a;q)}{1-aq^{2n-2}}   \Big), \nonumber \\
\widetilde{\beta}_n(a/q;q)&=\beta_n(a;q).
\end{align}
These two transformation formulas played important roles in proving Nahm sum identities via Bailey pairs (see e.g.,~\cite[(2.40), Lemma 2.3]{Cao-Wang2024} and \cite[Lemma 2.1]{Wang-Wang}).

\section{Proofs of the theorems}\label{sec-proof}
 Let $P_2(t)=t^2-t+\frac{1}{6}$ be the second periodic Bernoulli polynomial. It is well known that $q^{\frac{m}{24}}J_m$ and $q^{\frac{m}{24}+\frac{m}{2}P_2(\frac{a}{m})} J_{a,m}$ are modular forms of weight $1/2$ (see \eqref{J-defn}). Hence the last assertions in Theorems \ref{thm-main-1} and \ref{thm-main-2} follow easily from \eqref{main-1}, \eqref{eq-thm2-even}, \eqref{eq-thm2-odd} and direct calculations. Therefore, it suffices to prove the identities in these theorems.  We will prove them using the Bailey pair method. The initial steps of our proofs are inherited from our previous work \cite{Wang-Wang}.
\begin{proof}[Proof of Theorem \ref{thm-main-1}]
Note that
\begin{align}\label{eq-quadratic}
&n^\mathrm{T}\mathcal{C}(D_r)^{-1}n=\sum_{i=1}^{r-2} in_i^2+\frac{r}{4}(n_{r-1}^2+n_r^2)+2\sum_{1\leq i <j \leq r-2}in_in_j  \nonumber \\
&\quad \quad +\sum_{i=1}^{r-2}in_i(n_{r-1}+n_r) +(\frac{1}{2}r-1)n_{r-1} n_r \nonumber \\
&=\big(n_1+n_2+\cdots+n_{r-2}+\frac{n_{r-1}+n_r}{2}\big)^2+\cdots+\big(n_{r-2}+\frac{n_{r-1}+n_r}{2}\big)^2 \nonumber \\
&\quad \quad +\frac{1}{2}(n_{r-1}^2+n_r^2).
\end{align}
(1) We first consider the cases $\lambda \in \{0,1,2,\dots,\lfloor (r-1)/2\rfloor\}$. 

According to the parity of $n_{r-1}+n_r$, we divide our discussions into two cases.

\textbf{Case 1.} If $n_{r-1}+n_r$ is even, then we write 
\begin{align}\label{nk-even}
n_{r-1}=s_{r-1}+t, ~ n_{r}=s_{r-1}-t, ~ -s_{r-1}\leq t \leq s_{r-1}, ~s_{r-1}\in \mathbb{N}, 
\end{align}
and we introduce new variables:
\begin{equation}\label{eq-variable}
\begin{split}
n_1+n_2+\cdots +n_{r-2}+s_{r-1}&=s_1, \\
n_2+\cdots+n_{r-2}+s_{r-1}&=s_2, \\
\cdots  &\\
n_{r-2}+s_{r-1}&=s_{r-2}.
\end{split}
\end{equation}
In other words, we have
\begin{align}\label{eq-variable-relation}
n_1=s_1-s_2,~ n_2=s_2-s_3, ~\ldots,~  n_{r-2}=s_{r-2}-s_{r-1}.
\end{align}

From \eqref{eq-quadratic} and \eqref{nk-even} we have
\begin{align}\label{eq-main-even}
&\frac{1}{2}n^{\mathrm{T}}\mathcal{C}(D_{r})^{-1}n+\sum_{\ell=r-2\lambda +1}^{r-2}\Big\lfloor \frac{\ell +2\lambda +1-r}{2} \Big\rfloor n_{\ell}+\frac{\lambda}{2}(n_{r-1}+n_{r})\nonumber\\
&=\frac{1}{2}(s_1^2+s_2^2+\cdots+s_{r-2}^2+s_{r-1}^2+t^2)+\sum_{\ell=0}^{\lambda -1}s_{r-2\lambda+1+2\ell}.
\end{align}
Then
\begin{align}\label{main-S0-start}
    &S_0(q)\nonumber\\
    &=\sum_{s_1\geq s_2\geq \cdots \geq s_{r-2}\geq s_{r-1} \geq 0} \frac{q^{\frac{1}{2}(s_1^2+s_2^2+\cdots +s_{r-2}^2+s_{r-1}^2)+\sum_{\ell=0}^{\lambda -1}s_{r-2\lambda+1+2\ell}}}{(q;q)_{s_1-s_2}(q;q)_{s_2-s_3}\cdots (q;q)_{s_{r-3}-s_{r-2}}(q;q)_{s_{r-2}-s_{r-1}}} \nonumber \\
&\quad \quad \times \sum_{t=-s_{r-1}}^{s_{r-1}} \frac{q^{\frac{1}{2}t^2}}{(q;q)_{s_{r-1}+t}(q;q)_{s_{r-1}-t}} \nonumber \\
&=\sum_{s_1\geq s_2\geq \cdots \geq s_{r-2}\geq s_{r-1} \geq 0} \frac{q^{\frac{1}{2}(s_1^2+s_2^2+\cdots +s_{r-2}^2+s_{r-1}^2)+\sum_{\ell=0}^{\lambda -1}s_{r-2\lambda+1+2\ell}}}{(q;q)_{s_1-s_2}(q;q)_{s_2-s_3}\cdots (q;q)_{s_{r-3}-s_{r-2}}(q;q)_{s_{r-2}-s_{r-1}}}  \nonumber \\
&\quad \quad \times \frac{(-q^{\frac{1}{2}};q)_{s_{r-1}}^2}{(q;q)_{2s_{r-1}}}   \quad \text{(by \eqref{Andrews1})} \nonumber \\
&=\sum_{s_1\geq s_2\geq \cdots \geq s_{r-2}\geq s_{r-1} \geq 0} \frac{q^{\frac{1}{2}(s_1^2+s_2^2+\cdots +s_{r-2}^2+s_{r-1}^2)+\sum_{\ell=0}^{\lambda -1}s_{r-2\lambda+1+2\ell}}}{(q;q)_{s_1-s_2}(q;q)_{s_2-s_3}\cdots (q;q)_{s_{r-3}-s_{r-2}}(q;q)_{s_{r-2}-s_{r-1}}} \nonumber \\
&\quad \quad \times (-q^{\frac{1}{2}};q)_{s_{r-1}}\beta^{(1)}_{s_{r-1}}(1;q) \nonumber \\
&=\sum_{s_1\geq s_2\geq \cdots \geq s_{r-2\lambda-1} \geq 0} \frac{q^{\frac{1}{2}(s_1^2+s_2^2+\cdots +s_{r-2\lambda-1}^2)}(-q^{1/2};q)_{s_{r-2\lambda-1}}}{(q;q)_{s_1-s_2}(q;q)_{s_2-s_3}\cdots (q;q)_{s_{r-2\lambda-2}-s_{r-2\lambda-1}}}\nonumber \\
&\quad \quad \times \sum_{s_{r-2\lambda}=0}^{s_{r-2\lambda-1}}\frac{(-q^{3/2};q)_{s_{r-2\lambda}}q^{\frac{1}{2}s_{r-2\lambda}^2}}{(q;q)_{s_{r-2\lambda-1}-s_{r-2\lambda}}(-q^{1/2};q)_{s_{r-2\lambda-1}}}\nonumber\\
&\quad \quad \times \sum_{s_{r-2\lambda+1}=0}^{s_{r-2\lambda}}\frac{(-q^{1/2};q)_{s_{r-2\lambda+1}}q^{\frac{1}{2}s_{r-2\lambda+1}^2+s_{r-2\lambda+1}}}{(q;q)_{s_{r-2\lambda}-s_{r-2\lambda+1}}(-q^{3/2};q)_{s_{r-2\lambda}}} \times \cdots\nonumber\\
&\quad \quad \times \sum_{s_{r-2}=0}^{s_{r-3}}\frac{(-q^{3/2};q)_{s_{r-2}}q^{\frac{1}{2}s_{r-2}^2}}{(q;q)_{s_{r-3}-s_{r-2}}(-q^{1/2};q)_{s_{r-3}}}  \sum_{s_{r-1}=0}^{s_{r-2}}\frac{(-q^{1/2};q)_{s_{r-1}}q^{\frac{1}{2}s_{r-1}^2+s_{r-1}}}{(q;q)_{s_{r-2}-s_{r-1}}(-q^{3/2};q)_{s_{r-2}}} \nonumber\\
&\quad \quad \times \beta^{(1)}_{s_{r-1}}(1;q).
\end{align}
Here $(\alpha_n^{(1)}(1;q),\beta_n^{(1)}(1;q))$ is the Bailey pair with \cite[p.~470]{Slater1951}
\begin{equation}\label{470}
\begin{split}
\alpha_n^{(1)}(1;q)&=\begin{cases}
    1, & n=0, \\
    (-1)^mq^{3m^2}(q^{m/2}+q^{-m/2}), & n=2m, \\
    (-1)^{m+1}q^{3m^2}(q^{1+7m/2}-q^{(5m+1)/2}), & n=2m+1,
\end{cases} \\
    \beta_n^{(1)}(1;q)&=\frac{1}{(q^{1/2};q)_n(q^2;q^2)_n}.
\end{split}
\end{equation}
Applying \eqref{eq-BP-lift} to it, we obtain a Bailey pair relative to $q$:
\begin{align}
\begin{split}
\alpha_n^{(2)}(q;q)&=\begin{cases}
    1, & n=0, \\
    \frac{1-q^{4m+1}}{1-q}(-1)^mq^{3m^2-m/2}, & n=2m, \\
    \frac{1-q^{4m+3}}{1-q}(-1)^mq^{3m^2+5m/2+1/2}, & n=2m+1,
\end{cases} \\
\beta_n^{(2)}(q;q)&=\beta_n^{(1)}(1;q).
\end{split}
\end{align}
We obtain the Bailey pairs $(\alpha_n^{(2+2i)}(q;q),\beta_n^{(2+2i)}(q;q))$ by applying \eqref{S3}\eqref{S4} with $a=q$ (see \eqref{S3-aq} and \eqref{S4-aq}) to it and iterating $i$ times ($i=1,2,\dots, \lambda$), where
\begin{align}\label{34}
    \alpha_n^{(2+2i)}(q;q)&=\begin{cases}
    1, & n=0, \\
    \frac{1-q^{4m+1}}{1-q}(-1)^mq^{(4i+3)m^2+(2i-1/2)m}, & n=2m, \\
    \frac{1-q^{4m+3}}{1-q}(-1)^mq^{(4i+3)m^2+(6i+5/2)m+2i+1/2}, & n=2m+1.
\end{cases}
\end{align}
Substituting these Bailey pairs into \eqref{main-S0-start}, we deduce that
\begin{align}\label{main-S0-mid}
&S_0(q)\nonumber\\
&=\sum_{s_1\geq s_2\geq \cdots \geq s_{r-2\lambda-1} \geq 0}  \frac{q^{\frac{1}{2}(s_1^2+s_2^2+\cdots +s_{r-2\lambda-1}^2)}(-q^{1/2};q)_{s_{r-2\lambda-1}}}{(q;q)_{s_1-s_2}(q;q)_{s_2-s_3}\cdots (q;q)_{s_{r-2\lambda-2}-s_{r-2\lambda-1}}}\nonumber\\
&\quad\quad\times\beta^{(2+2\lambda)}_{s_{r-2\lambda-1}}(q;q)\nonumber\\
&=\sum_{s_1\geq0}q^{\frac{1}{2}s_1^2}(-q^{1/2};q)_{s_1}\sum_{s_2=0}^{s_1}\frac{q^{\frac{1}{2}s_2^2}(-q^{1/2};q)_{s_2}}{(q;q)_{s_1-s_2} (-q^{1/2};q)_{s_1}}\sum_{s_3=0}^{s_2}\frac{q^{\frac{1}{2}s_3^2}(-q^{1/2};q)_{s_3}}{(q;q)_{s_2-s_3} (-q^{1/2};q)_{s_2}} \times \cdots\nonumber\\
&\quad\quad\times \sum_{s_{r-2\lambda-1}=0}^{s_{r-2\lambda-2}} \frac{q^{\frac{1}{2}s_{r-2\lambda-1}^2}(-q^{1/2};q)_{s_{r-2\lambda-1}}}{(q;q)_{s_{r-2\lambda-2}-s_{r-2\lambda-1}} (-q^{1/2};q)_{s_{r-2\lambda-2}}} \beta^{(2+2\lambda)}_{s_{r-2\lambda-1}}(q;q).
\end{align}
Applying \eqref{eq-BP-reduce} to $(\alpha^{(2+2\lambda)}_{n}(q;q),\beta^{(2+2\lambda)}_{n}(q;q))$, we obtain a Bailey pair relative to $1$:
\begin{align}\label{reduce-S0}
\begin{split}
\alpha_n^{(3+2\lambda)}(1;q)&=\begin{cases}
    1, & n=0, \\
    (-1)^m (q^{(4\lambda+3)m^2+(2\lambda-1/2)m}+q^{(4\lambda+3)m^2-(2\lambda-1/2)m}), & n=2m, \\
    (-1)^m (q^{(4\lambda+3)m^2+(6\lambda+5/2)m+2\lambda+1/2}-q^{(4\lambda+3)m^2+(2\lambda+7/2)m+1}), & n=2m+1,
\end{cases}\\
\beta_n^{(3+2\lambda)}(1;q)&=\beta_n^{(2+2\lambda)}(q;q).
\end{split}
\end{align}
We obtain the Bailey pairs $(\alpha_n^{(3+2\lambda+i)}(1;q),\beta_n^{(3+2\lambda+i)}(1;q))$ by applying \eqref{S2} with $a=1$ (see \eqref{BP-S2-a1}) to it and iterating $i$ times ($i=1,2,\dots, r-2\lambda-2$), where
\begin{align}
    \alpha_n^{(3+2\lambda+i)}(1;q)=\begin{cases}
    1, & n=0, \\
    (-1)^m (q^{(4\lambda+2i+3)m^2+(2\lambda-1/2)m}+q^{(4\lambda+2i+3)m^2-(2\lambda-1/2)m}), & n=2m, \\
    \begin{aligned}
    &(-1)^m(q^{(4\lambda+2i+3)m^2+(6\lambda+2i+5/2)m+2\lambda+(i+1)/2}\\
    &\quad-q^{(4\lambda+2i+3)m^2+(2\lambda+2i+7/2)m+(i+2)/2})
    \end{aligned}, & n=2m+1.
\end{cases}
\end{align}
Substituting these Bailey pairs into \eqref{main-S0-mid}, we deduce that
\begin{align}\label{main-S0-end}
   S_0(q)&=\sum_{s_1\geq0}q^{\frac{1}{2}s_1^2}(-q^{1/2};q)_{s_1}\beta^{(r+1)}_{s_1}(1;q)\nonumber\\
    &=\frac{(-q^{1/2};q)_{\infty}}{(q;q)_{\infty}}\sum_{n\geq0}q^{n^2/2}\alpha^{(r+1)}_n (1;q)\quad \text{(by \eqref{BP-S2-a1-infinite})}\nonumber\\
    &=\frac{(-q^{1/2};q)_{\infty}}{(q;q)_{\infty}}\Big(1+\sum_{m\geq 1}(-1)^m\big(q^{(2r+1)m^2+(2\lambda-1/2)m}+q^{(2r+1)m^2-(2\lambda-1/2)m}\big)\nonumber\\
    &\qquad+\sum_{m\geq0}(-1)^m \big(q^{(2r+1)m^2+(2r+2\lambda+1/2)m+r/2+\lambda}-q^{(2r+1)m^2+(2r-2\lambda+3/2)m+r/2-\lambda+1/2}\big)\Big)\nonumber\\
    &=\frac{(-q^{1/2};q)_{\infty}}{(q;q)_{\infty}} \Big(\sum_{m=-\infty}^{\infty}(-1)^m q^{(2r+1)m^2+(2\lambda-1/2)m} \nonumber \\
    &\qquad -\sum_{m=-\infty}^{\infty} (-1)^m q^{(2r+1)m^2+(2r-2\lambda+3/2)m+r/2-\lambda+1/2}\Big) \nonumber\\
    &=\frac{(-q^{1/2};q)_{\infty}}{(q;q)_{\infty}} \Big((q^{2r+2\lambda+1/2},q^{2r-2\lambda+3/2},q^{4r+2};q^{4r+2})\nonumber\\
    &\quad\quad  -q^{(r+1)/2-\lambda}(q^{4r-2\lambda+5/2},q^{2\lambda-1/2},q^{4r+2};q^{4r+2})\Big).
\end{align}

\textbf{Case 2.} If $n_{r-1}\not\equiv n_{r}$ (mod 2), then we write 
\begin{align}
\label{odd-nk}
n_{r-1}=s_{r-1}+t+1, ~ n_{r}=s_{r-1}-t, ~ -s_{r-1}-1\leq t \leq s_{r-1},~ s_{r-1}\in \mathbb{N},
\end{align}
and we introduce the variables $s_1,s_2,\dots,s_{r-2}$ as in \eqref{eq-variable} and \eqref{eq-variable-relation}.
We have
\begin{align}\label{main-S1-start}
    &S_1(q)\nonumber\\
    &=q^{\frac{r+4\lambda}{8}}\sum_{s_1\geq s_2\geq \cdots\geq s_{r-1} \geq 0} \frac{q^{\frac{1}{2}(s_1^2+s_2^2+\cdots +s_{r-1}^2+s_1+s_2+\cdots+s_{r-1})+\sum_{\ell=0}^{\lambda -1}s_{r-2\lambda+1+2\ell}}}{(q;q)_{s_1-s_2}(q;q)_{s_2-s_3}\cdots (q;q)_{s_{r-2}-s_{r-1}}} \nonumber \\
&\quad \quad \times \sum_{t=-s_{r-1}-1}^{s_{r-1}} \frac{q^{\frac{1}{2}t^2+\frac{1}{2}t}}{(q;q)_{s_{r-1}+t+1}(q;q)_{s_{r-1}-t}} \nonumber \\
&=q^{\frac{r+4\lambda}{8}}\sum_{s_1\geq s_2\geq \cdots\geq s_{r-1} \geq 0} \frac{q^{\frac{1}{2}(s_1^2+s_2^2+\cdots +s_{r-1}^2+s_1+s_2+\cdots +s_{r-1})+\sum_{\ell=0}^{\lambda -1}s_{r-2\lambda+1+2\ell}}}{(q;q)_{s_1-s_2}(q;q)_{s_2-s_3}\cdots (q;q)_{s_{r-2}-s_{r-1}}} \nonumber \\
&\quad \quad \times \frac{(-q;q)_{s_{r-1}}(-1;q)_{s_{r-1}+1}}{(q;q)_{2s_{r-1}+1}} \quad \text{(by \eqref{Andrews2})} \nonumber \\
&=\frac{2q^{\frac{r+4\lambda}{8}}}{1-q}\sum_{s_1\geq s_2\geq \cdots\geq s_{r-1} \geq 0} \frac{q^{\frac{1}{2}(s_1^2+s_2^2+\cdots +s_{r-1}^2+s_1+s_2+\cdots +s_{r-1})+\sum_{\ell=0}^{\lambda -1}s_{r-2\lambda+1+2\ell}}}{(q;q)_{s_1-s_2}(q;q)_{s_2-s_3}\cdots (q;q)_{s_{r-2}-s_{r-1}}} \nonumber \\
&\quad \quad \times (-q;q)_{s_{r-1}}\beta^{(1)}_{s_{r-1}}(q;q) \nonumber \\
&=\frac{2q^{\frac{r+4\lambda}{8}}}{1-q}\sum_{s_1\geq s_2\geq \cdots \geq s_{r-2\lambda-1} \geq 0} \frac{q^{\frac{1}{2}(s_1^2+s_2^2+\cdots +s_{r-2\lambda-1}^2+s_1+s_2+\cdots +s_{r-2\lambda-1})}(-q;q)_{s_{r-2\lambda-1}}}{(q;q)_{s_1-s_2}(q;q)_{s_2-s_3}\cdots (q;q)_{s_{r-2\lambda-2}-s_{r-2\lambda-1}}}\nonumber \\
&\quad \quad \times \sum_{s_{r-2\lambda}=0}^{s_{r-2\lambda-1}}\frac{(-q^{2};q)_{s_{r-2\lambda}}q^{\frac{1}{2}s_{r-2\lambda}^2+\frac{1}{2}s_{r-2\lambda}}}{(q;q)_{s_{r-2\lambda-1}-s_{r-2\lambda}}(-q;q)_{s_{r-2\lambda-1}}}\nonumber\\
&\quad \quad \times \sum_{s_{r-2\lambda+1}=0}^{s_{r-2\lambda}}\frac{(-q;q)_{s_{r-2\lambda+1}}q^{\frac{1}{2}s_{r-2\lambda+1}^2+\frac{3}{2}s_{r-2\lambda+1}}}{(q;q)_{s_{r-2\lambda}-s_{r-2\lambda+1}}(-q^{2};q)_{s_{r-2\lambda}}} \times \cdots\nonumber\\
&\quad \quad \times \sum_{s_{r-2}=0}^{s_{r-3}}\frac{(-q^{2};q)_{s_{r-2}}q^{\frac{1}{2}s_{r-2}^2+\frac{1}{2}s_{r-2}}}{(q;q)_{s_{r-3}-s_{r-2}}(-q;q)_{s_{r-3}}}  \sum_{s_{r-1}=0}^{s_{r-2}}\frac{(-q;q)_{s_{r-1}}q^{\frac{1}{2}s_{r-1}^2+\frac{3}{2}s_{r-1}}}{(q;q)_{s_{r-2}-s_{r-1}}(-q^{2};q)_{s_{r-2}}} \nonumber\\
&\quad \quad \times \beta^{(1)}_{s_{r-1}}(q;q).
\end{align}
Here $(\alpha_n^{(1)}(q;q),\beta_n^{(1)}(q;q))$ is the Bailey pair with \cite[C(3)]{Slater1951}
\begin{equation}\label{C3}
\begin{split}
\alpha_n^{(1)}(q;q)&=\begin{cases}
    1, & n=0, \\
    (-1)^mq^{3m^2+m}, & n=2m, \\
    (-1)^{m+1}q^{3m^2+5m+2}, & n=2m+1,
\end{cases} \\
    \beta_n^{(1)}(q;q)&=\frac{1}{(q;q)_n(q^3;q^2)_n}.
\end{split}
\end{equation}
Applying \eqref{eq-BP-lift} to it, we obtain a Bailey pair relative to $q^2$:
\begin{align}
\begin{split}
\alpha_n^{(2)}(q^2;q)&=\begin{cases}
    1, & n=0, \\
    \frac{1-q^{4m+2}}{1-q^2}(-1)^mq^{3m^2+m}, & n=2m, \\
    0, & n=2m+1,
\end{cases} \\
\beta_n^{(2)}(q^2;q)&=\beta_n^{(1)}(q;q).
\end{split}
\end{align}
We obtain the Bailey pairs $(\alpha_n^{(2+2i)}(q^2;q),\beta_n^{(2+2i)}(q^2;q))$ by applying \eqref{S6}\eqref{S5} with $a=q^2$ to it and iterating $i$ times ($i=1,2,\dots, \lambda$), where
\begin{align}\label{65}
    \alpha_n^{(2+2i)}(q^2;q)&=\begin{cases}
    1, & n=0, \\
    \frac{1-q^{4m+2}}{1-q^2}(-1)^mq^{(4i+3)m^2+(4i+1)m}, & n=2m, \\
    0, & n=2m+1.
\end{cases}
\end{align}
Substituting these Bailey pairs into \eqref{main-S1-start}, we deduce that
\begin{align}\label{main-S1-mid}
&S_1(q)\nonumber\\
&=\frac{2q^{\frac{r+4\lambda}{8}}}{1-q}\sum_{s_1\geq s_2\geq \cdots \geq s_{r-2\lambda-1} \geq 0} \frac{q^{\frac{1}{2}(s_1^2+s_2^2+\cdots +s_{r-2\lambda-1}^2+s_1+s_2+\cdots +s_{r-2\lambda-1})}(-q;q)_{s_{r-2\lambda-1}}}{(q;q)_{s_1-s_2}(q;q)_{s_2-s_3}\cdots (q;q)_{s_{r-2\lambda-2}-s_{r-2\lambda-1}}}\nonumber\\
&\quad\quad\times\beta^{(2+2\lambda)}_{s_{r-2\lambda-1}}(q^2;q)\nonumber\\
&=\frac{2q^{\frac{r+4\lambda}{8}}}{1-q}\sum_{s_1\geq0}q^{\frac{1}{2}s_1^2+\frac{1}{2}s_1}(-q;q)_{s_1}\sum_{s_2=0}^{s_1}\frac{q^{\frac{1}{2}s_2^2+\frac{1}{2}s_2}(-q;q)_{s_2}}{(q;q)_{s_1-s_2} (-q;q)_{s_1}} \times \cdots\nonumber\\
&\quad\quad\times \sum_{s_{r-2\lambda-1}=0}^{s_{r-2\lambda-2}} \frac{q^{\frac{1}{2}s_{r-2\lambda-1}^2+\frac{1}{2}s_{r-2\lambda-1}}(-q;q)_{s_{r-2\lambda-1}}}{(q;q)_{s_{r-2\lambda-2}-s_{r-2\lambda-1}} (-q;q)_{s_{r-2\lambda-2}}} \beta^{(2+2\lambda)}_{s_{r-2\lambda-1}}(q^2;q).
\end{align}
Applying \eqref{eq-BP-reduce} to $(\alpha^{(2+2\lambda)}_{n}(q^2;q),\beta^{(2+2\lambda)}_{n}(q^2;q))$, we obtain a Bailey pair relative to $q$:
\begin{align}
\begin{split}
\alpha_n^{(3+2\lambda)}(q;q)&=\begin{cases}
    1, & n=0, \\
    (-1)^mq^{(4\lambda+3)m^2+(4\lambda+1)m}, & n=2m, \\
    (-1)^{m+1} q^{(4\lambda+3)m^2+(4\lambda+5)m+2}, & n=2m+1,
\end{cases}\\
\beta_n^{(3+2\lambda)}(q;q)&=\beta_n^{(2+2\lambda)}(q^2;q).
\end{split}
\end{align}
We obtain the Bailey pairs $(\alpha_n^{(3+2\lambda+i)}(q;q),\beta_n^{(3+2\lambda+i)}(q;q))$ by applying \eqref{S2} with $a=q$ (see \eqref{BP-S2-aq}) to it and iterating $i$ times ($i=1,2,\dots, r-2\lambda-2$), where
\begin{align}
    \alpha_n^{(3+2\lambda+i)}(q;q)&=\begin{cases}
    1, & n=0, \\
    (-1)^mq^{(4\lambda+2i+3)m^2+(4\lambda+i+1)m}, & n=2m, \\
    (-1)^{m+1} q^{(4\lambda+2i+3)m^2+(4\lambda+3i+5)m+i+2}, & n=2m+1.
\end{cases}
\end{align}
Substituting these Bailey pairs into \eqref{main-S1-mid}, we deduce that
\begin{align}\label{main-S1-end}
    S_1(q)
    &=\frac{2q^{\frac{r+4\lambda}{8}}}{1-q}\sum_{s_1\geq0}q^{\frac{1}{2}s_1^2+\frac{1}{2}s_1}(-q;q)_{s_1}\beta^{(r+1)}_{s_1}(q;q)\nonumber\\
    &=\frac{2q^{\frac{r+4\lambda}{8}}}{1-q}\frac{(-q;q)_{\infty}}{(q^2;q)_{\infty}}\sum_{n\geq0}q^{\frac{1}{2}n^2+\frac{1}{2}n}\alpha^{(r+1)}_n (q;q) \quad\text{(by \eqref{BP-S2-aq-infinite})}\nonumber\\
    &=2q^{\frac{r+4\lambda}{8}}\frac{(-q;q)_{\infty}}{(q;q)_{\infty}}\Big(\sum_{m\geq0}(-1)^mq^{(2r+1)m^2+(r+2\lambda)m}\nonumber\\
    &\qquad+\sum_{m\geq0}(-1)^{m+1}q^{(2r+1)m^2+(3r-2\lambda+2)m+r-2\lambda+1} \Big)\nonumber\\
    &=2q^{\frac{r+4\lambda}{8}}\frac{(-q;q)_{\infty}}{(q;q)_{\infty}}\sum_{m=-\infty}^{\infty}(-1)^m q^{(2r+1)m^2+(r+2\lambda)m} \nonumber\\
    &=2q^{\frac{r+4\lambda}{8}}\frac{(-q;q)_{\infty}}{(q;q)_{\infty}} (q^{3r+2\lambda+1},q^{r-2\lambda+1},q^{4r+2};q^{4r+2}).
\end{align}
Adding \eqref{main-S0-end} and \eqref{main-S1-end} together, we deduce \eqref{main-1} for $\lambda \in \{0,1,2,\dots,\lfloor (r-1)/2\rfloor\}$.

(2) We now discuss the remaining case with the rank $r=2k$ and $\lambda=k$. The arguments are almost the same with part (1) except that the Bailey pairs we used at the last step are different.

We first consider the case when $n_{2k-1}+n_{2k}$ is even. We obtain the Bailey pair $(\alpha^{(2k)}_n,\beta^{(2k)}_n)$ from \eqref{34}, where
\begin{align}
    \alpha_n^{(2k)}(q;q)&=\begin{cases}
    1, & n=0, \\
    \frac{1-q^{4m+1}}{1-q}(-1)^mq^{(4k-1)m^2+(2k-5/2)m}, & n=2m, \\
    \frac{1-q^{4m+3}}{1-q}(-1)^mq^{(4k-1)m^2+(6k-7/2)m+2k-3/2}, & n=2m+1.
\end{cases}
\end{align}
Arguing similarly as in (1), we have
\begin{align}\label{2kk0}
    S_0(q)&=\sum_{s_1\ge0}q^{\frac{1}{2}s_1^2+s_1}(-q^{\frac{1}{2}};q)_{s_1}\beta^{(2k)}_{s_1}(q;q)\nonumber\\
    &=\frac{(-q^{\frac{1}{2}};q)_\infty}{(q^2;q)_\infty} \sum_{n=0}^\infty \frac{q^{\frac{1}{2}n^2+n}}{1+q^{n+\frac{1}{2}}}\alpha^{(2k)}_n(q;q)\quad\text{(by \eqref{S3-aq-infinite})}\nonumber\\
    &=\frac{(-q^{\frac{1}{2}};q)_\infty}{(q;q)_\infty}\Big(\sum_{m\ge 0} (1-q^{2m+\frac{1}{2}})(-1)^mq^{(4k+1)m^2+(2k-\frac{1}{2})m}\nonumber\\
    &\quad\quad\quad\quad\quad\quad\quad+\sum_{m\ge 0} (1-q^{2m+\frac{3}{2}})(-1)^mq^{(4k+1)m^2+(6k+\frac{1}{2})m+2k}\Big)\nonumber\\
    &=\frac{(-q^{\frac{1}{2}};q)_\infty}{(q;q)_\infty}\Big(\sum_{m\ge0}(-1)^m\big(q^{(4k+1)m^2+(2k-\frac{1}{2})m}-q^{(4k+1)m^2+(6k+\frac{5}{2})m+2k+\frac{3}{2}}\big)\nonumber\\
     &\quad\quad\quad\quad\quad\quad\quad-\sum_{m\ge0}(-1)^m\big(q^{(4k+1)m^2+(2k+\frac{3}{2})m+\frac{1}{2}}-q^{(4k+1)m^2+(6k+\frac{1}{2})m+2k}\big)\Big)\nonumber\\
     &=\frac{(-q^{\frac{1}{2}};q)_\infty}{(q;q)_\infty}\Big(\sum_{m=-\infty}^{\infty}(-1)^mq^{(4k+1)m^2+(2k-\frac{1}{2})m}-\sum_{m=-\infty}^{\infty}(-1)^mq^{(4k+1)m^2+(2k+\frac{3}{2})m+\frac{1}{2}}\Big)\nonumber\\
     &=\frac{(-q^{\frac{1}{2}};q)_\infty}{(q;q)_\infty}((q^{6k+\frac{1}{2}},q^{2k+\frac{3}{2}},q^{8k+2};q^{8k+2})_{\infty}-q^{\frac{1}{2}}(q^{6k+\frac{5}{2}},q^{2k-\frac{1}{2}},q^{8k+2};q^{8k+2})_{\infty}).
\end{align}

We then consider the case when $n_{2k-1}+ n_{2k}$  is odd. We obtain the Bailey pair $(\alpha^{(2k)}_n,\beta^{(2k)}_n)$ from \eqref{65}, where
\begin{align}
    \alpha_n^{(2k)}(q^2;q)&=\begin{cases}
    1, & n=0, \\
    \frac{1-q^{4m+2}}{1-q^2}(-1)^mq^{(4k-1)m^2+(4k-3)m}, & n=2m, \\
    0, & n=2m+1.
\end{cases}
\end{align}
Then we have
\begin{align}\label{2kk1}
    S_1(q)&=\frac{2q^{\frac{3}{4}k}}{1-q}\sum_{s_1\ge 0}(-q;q)_{s_1}q^{\frac{1}{2}s_1^2+\frac{3}{2}s_1}\beta_{s_1}^{(2k)}(q^2;q)\nonumber\\
    &=\frac{2q^{\frac{3}{4}k}}{1-q}\frac{(-q^2;q)_{\infty}}{(q^3;q)_{\infty}}\sum_{n=0}^{\infty}\frac{1+q}{1+q^{n+1}}q^{\frac{1}{2}n^2+\frac{3}{2}n}\alpha_{n}^{(2k)}(q^2;q)\nonumber\\
    &=2q^{\frac{3}{4}k}\frac{(-q;q)_{\infty}}{(q;q)_{\infty}}\sum_{m=0}^{\infty}(-1)^m(1-q^{2m+1})q^{(4k+1)m^2+4km}\nonumber\\
    &=2q^{\frac{3}{4}k}\frac{(-q;q)_{\infty}}{(q;q)_{\infty}}\sum_{m=-\infty}^{\infty}(-1)^mq^{(4k+1)m^2+4km}\nonumber\\
    &=2q^{\frac{3}{4}k}\frac{(-q;q)_{\infty}}{(q;q)_{\infty}} (q^{8k+1},q,q^{8k+2};q^{8k+2})_{\infty}.
\end{align}
Here for the second equality we used the limiting case of \eqref{S6} with $a=q^2$. 

Adding \eqref{2kk0} and \eqref{2kk1} together, we obtain \eqref{main-1} for $r=2k$ and $\lambda=k$.

Combining (1) and (2), we prove \eqref{main-1}.
\end{proof}

\begin{proof}[Proof of Theorem \ref{thm-main-2}]
(1) We first consider the even rank case $r=2k$. From \eqref{eq-quadratic} we have
\begin{align}\label{eq-quadratic-even}
&\frac{1}{2}n^\mathrm{T}\mathcal{C}(D_{2k})^{-1}n+n^\mathrm{T}B^{(0)}
=\frac{1}{2}(s_1^2+s_2^2+\cdots+s_{2k-2}^2+s_{2k-1}^2+t^2) \nonumber \\
&\qquad +\frac{1}{2}(s_1-s_2+s_3-s_4+\cdots+s_{2k-3}-s_{2k-2})+\frac{1}{2}(s_{2k-1}-t).
\end{align}

When $n_{2k-1}+n_{2k}$ is even, using \eqref{nk-even} and \eqref{Andrews1} we have 
\begin{align}
    S_0(q)&=\sum_{s_1\geq s_2\geq \cdots \geq s_{2k-2}\geq s_{2k-1}\geq 0} \frac{q^{\frac{1}{2}(s_1^2+s_2^2+\cdots+s_{2k-2}^2+s_{2k-1}^2)+\frac{1}{2}(s_1-s_2+\cdots+s_{2k-3}-s_{2k-2}+s_{2k-1})}}{(q;q)_{s_1-s_2}(q;q)_{s_2-s_3}\cdots (q;q)_{s_{2k-3}-s_{2k-2}}(q;q)_{s_{2k-2}-s_{2k-1}}} \nonumber \\
    &\qquad \times \sum_{t=-s_{2k-1}}^{s_{2k-1}} \frac{q^{\frac{1}{2}(t^2-t)}}{(q;q)_{s_{2k-1}-t}(q;q)_{s_{2k-1}+t}} \nonumber \\
     &=\sum_{s_1=0}^\infty q^{\frac{1}{2}(s_1^2+s_1)}(-1;q)_{s_1}\sum_{s_2=0}^{s_1}\frac{q^{\frac{1}{2}(s_2^2-s_2)}(-q;q)_{s_2}}{(q;q)_{s_1-s_2}(-1;q)_{s_1}}\cdots \nonumber \\
    &\qquad \sum_{s_{2k-3}=0}^{s_{2k-4}}\frac{q^{\frac{1}{2}(s_{2k-3
    }^2+s_{2k-3})}(-1;q)_{s_{2k-3}}}{(q;q)_{s_{2k-4}-s_{2k-3}}(-q;q)_{s_{2k-4}}}
    \sum_{s_{2k-2}=0}^{s_{2k-3}}\frac{q^{\frac{1}{2}(s_{2k-2
    }^2-s_{2k-2})}(-q;q)_{s_{2k-2}}}{(q;q)_{s_{2k-3}-s_{2k-2}}(-1;q)_{s_{2k-3}}} \nonumber \\
    &\qquad     \sum_{s_{2k-1}=0}^{s_{2k-2}}\frac{q^{\frac{1}{2}(s_{2k-1}^2+s_{2k-1})}(-1;q)_{s_{2k-1}}}{(q;q)_{s_{2k-2}-s_{2k-1}}(-q;q)_{s_{2k-2}}}\beta_{s_{2k-1}}^{(1)}(1;q).
\end{align}
Here $(\alpha_n^{(1)}(1;q),\beta_n^{(1)}(1;q))$ is the Bailey pair with \cite[(C1)]{Slater1951}
\begin{equation}\label{C1}
\begin{split}
\alpha_n^{(1)}(1;q)&=\begin{cases}
    1, & n=0, \\
    (-1)^mq^{3m^2}(q^m+q^{-m}), & n=2m, \\
    0, & n=2m+1,
\end{cases} \\
    \beta_n^{(1)}(1;q)&=\frac{1}{(q;q)_n(q;q^2)_n}.
\end{split}
\end{equation}

Applying \eqref{S5-special} and \eqref{S6-special} to it and iterating $i$ times ($i=1,2,\dots,k-1$), we obtain the Bailey pairs $(\alpha_n^{(1+2i)}(1;q),\beta_n^{(1+2i)}(1;q))$ with
\begin{align}
    \alpha_n^{(1+2i)}(1;q)=q^{in^2}\alpha_n^{(1)}(1;q).
\end{align}
We have 
\begin{align}\label{proof-thm2-even-S0}
    S_0(q)&=\sum_{s_1=0}^\infty q^{\frac{1}{2}(s_1^2+s_1)}(-1;q)_{s_1}\beta_{s_1}^{(2k-1)}(1;q) \nonumber \\
    &=\frac{(-q;q)_\infty}{(q;q)_\infty} \sum_{n=0}^\infty \frac{q^{n(n+1)/2}(-1;q)_n}{(-q;q)_n}\alpha_n^{(2k-1)}(1;q) \quad\text{(by \eqref{S6-a1-infinite})}\nonumber \\
    &=\frac{(-q;q)_\infty}{(q;q)_\infty}\Big(1+\sum_{n=1}^\infty \frac{q^{n(n+1)/2}(-1;q)_n}{(-q;q)_n}q^{(k-1)n^2}\alpha_n^{(1)}(1;q)  \Big) \nonumber \\
    &=\frac{(-q;q)_\infty}{(q;q)_\infty}\Big(1+2\sum_{m=1}^\infty \frac{q^{m(2m+1)}}{1+q^{2m}}q^{4(k-1)m^2}(-1)^mq^{3m^2-m}(1+q^{2m})\Big) \nonumber \\
    &=\frac{(-q;q)_\infty}{(q;q)_\infty}\sum_{m=-\infty}^\infty (-1)^mq^{(4k+1)m^2} \nonumber \\
    &=\frac{(-q;q)_\infty}{(q;q)_\infty}(q^{4k+1},q^{4k+1},q^{8k+2};q^{8k+2})_\infty.
\end{align}

When $n_{2k-1}+n_{2k}$ is odd, using \eqref{odd-nk} we have
\begin{align}
    S_1(q)&=q^{\frac{1}{4}k}\sum_{s_1\geq s_2\geq \cdots \geq s_{2k-2}\geq s_{2k-1}\geq 0} \frac{q^{\frac{1}{2}(s_1^2+\cdots+s_{2k-2}^2+s_{2k-1}^2)+s_1+s_3+\cdots +s_{2k-3}+s_{2k-1}}}{(q;q)_{s_1-s_2}\cdots (q;q)_{s_{2k-3}-s_{2k-2}}(q;q)_{s_{2k-2}-s_{2k-1}}}  \nonumber \\
    &\qquad \times \sum_{t=-s_{2k-1}-1}^{s_{2k-1}} \frac{q^{\frac{1}{2}t^2}}{(q;q)_{s_{2k-1}-t}(q;q)_{s_{2k-1}+t+1}} \nonumber \\
    &=q^{\frac{1}{4}k}\sum_{s_1\geq s_2\geq \cdots \geq s_{2k-2}\geq s_{2k-1}\geq 0} \frac{q^{\frac{1}{2}(s_1^2+\cdots+s_{2k-2}^2)+s_1+s_3+\cdots +s_{2k-3}}}{(q;q)_{s_1-s_2}\cdots (q;q)_{s_{2k-3}-s_{2k-2}}(q;q)_{s_{2k-2}-s_{2k-1}}} \nonumber \\
    &\qquad \times \frac{q^{\frac{1}{2}s_{2k-1}^2+s_{2k-1}}(-q^{\frac{1}{2}};q)_{s_{2k-1}}(-q^{\frac{1}{2}};q)_{s_{2k-1}+1}}{(q;q)_{2s_{2k-1}+1}}  \quad \text{(by \eqref{Andrews2})}
    \nonumber \\
    &=q^{\frac{1}{4}k}\sum_{s_1=0}^\infty q^{\frac{1}{2}s_1^2+s_1}(-q^{\frac{1}{2}};q)_{s_1} \sum_{s_2=0}^{s_1} \frac{q^{\frac{1}{2}s_2^2}(-q^{\frac{3}{2}};q)_{s_2}}{(q;q)_{s_1-s_2}(-q^{\frac{1}{2}};q)_{s_1}}\cdots  \nonumber \\
    &\qquad \times
    \sum_{s_{2k-2}=0}^{s_{2k-3}} \frac{q^{\frac{1}{2}s_{2k-2}^2}(-q^{\frac{3}{2}};q)_{s_{2k-2}}}{(q;q)_{s_{2k-3}-s_{2k-2}}(-q^{\frac{1}{2}};q)_{s_{2k-3}}} \nonumber \\
    &\qquad \times \sum_{s_{2k-1}=0}^{s_{2k-2}} \frac{q^{\frac{1}{2}s_{2k-1}^2+s_{2k-1}}(-q^{\frac{1}{2}};q)_{s_{2k-1}}}{(q;q)_{s_{2k-2}-s_{2k-1}}(-q^{\frac{3}{2}};q)_{s_{2k-2}}}\beta_{s_{2k-1}}^{(1)}(q;q).
\end{align}
Here $(\alpha_n^{(1)}(q;q),\beta_n^{(1)}(q;q))$ is the G(2) Bailey pair in \cite[p.~469]{Slater1951} with $q^{1/2}$ replaced by $-q^{1/2}$:
\begin{equation}\label{BP-new}
\begin{split}
    \alpha_n^{(1)}(q;q)&=
        (-1)^{\lfloor (n+1)/2\rfloor} q^{\frac{3}{4}n^2+\frac{1}{4}n}\frac{1+q^{n+\frac{1}{2}}}{1-q},  \\
    \beta_n^{(1)}(q;q)&=\frac{1}{(q^{\frac{1}{2}};q)_{n+1}(q^2;q^2)_n}.
\end{split}
\end{equation}
Applying \eqref{S3-aq} and \eqref{S4-aq} to it and iterating $i$ times ($i=1,2,\dots,k-1$), we obtain the Bailey pairs $(\alpha_n^{(1+2i)}(q;q),\beta_n^{(1+2i)}(q;q))$ with
\begin{align}
    \alpha_n^{(1+2i)}(q;q)=q^{i(n^2+n)}\alpha_n^{(1)}(q;q).
\end{align}
We have
\begin{align}\label{proof-thm2-even-S1}
  S_1(q)   
    &=q^{\frac{1}{4}k}\sum_{s_1=0}^\infty q^{\frac{1}{2}s_1^2+s_1}(-q^{\frac{1}{2}};q)_{s_1}\beta_{s_1}^{(2k-1)}(q;q) \nonumber \\
    &=q^{\frac{1}{4}k}\frac{(-q^{1/2};q)_\infty}{(q;q)_\infty} \sum_{n=0}^\infty \frac{q^{\frac{1}{2}n^2+n}}{1+q^{n+\frac{1}{2}}}\alpha_n^{(2k-1)}(q;q)\quad\text{(by \eqref{S3-aq-infinite})} \nonumber \\
    &=q^{\frac{1}{4}k}\frac{(-q^{1/2};q)_\infty}{(q;q)_\infty} \Big(\sum_{n=0}^\infty \frac{q^{\frac{1}{2}n^2+n}}{1+q^{n+\frac{1}{2}}}q^{(k-1)(n^2+n)}\alpha_n^{(1)}(q;q)\Big) \nonumber \\
    &=q^{\frac{1}{4}k}\frac{(-q^{1/2};q)_\infty}{(q;q)_\infty} \Big(\sum_{n=0}^\infty q^{(k+\frac{1}{4})n^2+(k+\frac{1}{4})n}(-1)^{\lfloor \frac{n+1}{2}\rfloor} \Big) \nonumber \\
    &=q^{\frac{1}{4}k}\frac{(-q^{1/2};q)_\infty}{(q;q)_\infty}\sum_{n=-\infty}^\infty (-1)^nq^{(4k+1)n^2+(2k+\frac{1}{2})n} \nonumber \\
    &=q^{\frac{1}{4}k}\frac{(-q^{1/2};q)_\infty}{(q;q)_\infty}(q^{2k+\frac{1}{2}},q^{6k+\frac{3}{2}},q^{8k+2};q^{8k+2})_\infty.
\end{align}
Adding \eqref{proof-thm2-even-S0} and \eqref{proof-thm2-even-S1} together, we obtain \eqref{eq-thm2-even}.

(2) Now we consider the odd rank case $r=2k+1$.
From \eqref{eq-quadratic} we have
\begin{align}\label{eq-quadratic-odd}
&\frac{1}{2}n^\mathrm{T}\mathcal{C}(D_{2k+1})^{-1}n+n^\mathrm{T}B^{(1)}
=\frac{1}{2}(s_1^2+s_2^2+\cdots+s_{2k-1}^2+s_{2k}^2+t^2+t) \nonumber \\
&\qquad \qquad +\frac{1}{2}(-s_1+s_2-s_3+s_4-\cdots+s_{2k-2}-s_{2k-1}+s_{2k}).
\end{align}
 We have 
\begin{align}
    S_0(q)&=\sum_{s_1\geq s_2\geq \cdots \geq s_{2k-1}\geq s_{2k}\geq 0} \frac{q^{\frac{1}{2}(s_1^2+s_2^2+\cdots+s_{2k-1}^2+s_{2k}^2)+\frac{1}{2}(-s_1+s_2+\cdots+s_{2k-2}-s_{2k-1}+s_{2k})}}{(q;q)_{s_1-s_2}(q;q)_{s_2-s_3}\cdots (q;q)_{s_{2k-1}-s_{2k}}} \nonumber \\
    &\qquad \times \sum_{t=-s_{2k}}^{s_{2k}} \frac{q^{\frac{1}{2}(t^2+t)}}{(q;q)_{s_{2k}-t}(q;q)_{s_{2k}+t}} \nonumber \\
    &=\sum_{s_1=0}^\infty q^{\frac{1}{2}(s_1^2-s_1)}(-q;q)_{s_1}\sum_{s_2=0}^{s_1}\frac{q^{\frac{1}{2}(s_2^2+s_2)}(-1;q)_{s_2}}{(q;q)_{s_1-s_2}(-q;q)_{s_1}}\cdots \nonumber \\
    &\qquad \times \sum_{s_{2k-2}=0}^{s_{2k-3}}\frac{q^{\frac{1}{2}(s_{2k-2
    }^2+s_{2k-2})}(-1;q)_{s_{2k-2}}}{(q;q)_{s_{2k-3}-s_{2k-2}}(-q;q)_{s_{2k-3}}}
    \sum_{s_{2k-1}=0}^{s_{2k-2}}\frac{q^{\frac{1}{2}(s_{2k-1
    }^2-s_{2k-1})}(-q;q)_{s_{2k-1}}}{(q;q)_{s_{2k-2}-s_{2k-1}}(-1;q)_{s_{2k-2}}} \nonumber \\
    &\qquad  \times    \sum_{s_{2k}=0}^{s_{2k-1}}\frac{q^{\frac{1}{2}(s_{2k}^2+s_{2k})}(-1;q)_{s_{2k}}}{(q;q)_{s_{2k-1}-s_{2k}}(-q;q)_{s_{2k-1}}}\beta_{s_{2k}}^{(1)}(1;q).
\end{align}
Here we used \eqref{Andrews1} for the last equality, and $(\alpha_n^{(1)}(1;q),\beta_n^{(1)}(1;q))$ is the Bailey pair in \eqref{C1}. Applying \eqref{S5-special} and \eqref{S6-special} to this Bailey pair and iterating $i$ times ($i=1,2,\dots,k-1$), we obtain the Bailey pair $(\alpha_n^{(1+2i)}(1;q),\beta_n^{(1+2i)}(1;q))$ with
\begin{align}
    \alpha_n^{(1+2i)}(1;q)=q^{in^2}\alpha_n^{(1)}(1;q).
\end{align}
Applying \eqref{S6-special} to the Bailey pair $(\alpha_n^{(2k-1)}(1;q),\beta_n^{(2k-1)}(1;q))$, we obtain the Bailey pair $(\alpha_n^{(2k)}(1;q),\beta_n^{(2k)}(1;q))$ with
\begin{align}
    \alpha_n^{(2k)}(1;q)=\frac{(-1;q)_n}{(-q;q)_n}q^{(n^2+n)/2}\alpha_n^{(2k-1)}(1;q).
\end{align}
We have
\begin{align}\label{proof-thm2-odd-S0}
S_0(q)&=\sum_{s_1=0}^\infty q^{\frac{1}{2}(s_1^2-s_1)}(-q;q)_{s_1}\beta_{s_1}^{(2k)}(1;q)\nonumber \\
&=\frac{(-1;q)_\infty}{(q;q)_\infty}\sum_{n=0}^\infty \frac{(-q;q)_n}{(-1;q)_n}q^{(n^2-n)/2}\alpha_n^{(2k)}(1;q) \quad\text{(by \eqref{S5-a1-infinite})}\nonumber \\
&=\frac{(-1;q)_\infty}{(q;q)_\infty}\sum_{n=0}^\infty q^{kn^2}\alpha_n^{(1)}(1;q) \nonumber \\
&=\frac{(-1;q)_\infty}{(q;q)_\infty} \Big(1+\sum_{m=1}^\infty (-1)^mq^{(4k+3)m^2}(q^m+q^{-m})\Big) \nonumber \\
&=\frac{(-1;q)_\infty}{(q;q)_\infty} \sum_{m=-\infty}^\infty (-1)^mq^{(4k+3)m^2-m}\nonumber \\
&=2\frac{(-q;q)_\infty}{(q;q)_\infty}(q^{4k+2},q^{4k+4},q^{8k+6};q^{8k+6})_\infty.
\end{align}

Similarly, in view of \eqref{odd-nk}, we have 
\begin{align}
    S_1(q)&=q^{\frac{2k+3}{8}}\sum_{s_1\geq s_2\geq \cdots \geq s_{2k-1}\geq s_{2k}\geq 0} \frac{q^{\frac{1}{2}(s_1^2+s_2^2+\cdots+s_{2k-1}^2+s_{2k}^2)+(s_2+s_4+\cdots+s_{2k-2}+s_{2k})}}{(q;q)_{s_1-s_2}(q;q)_{s_2-s_3}\cdots (q;q)_{s_{2k-2}-s_{2k-1}}(q;q)_{s_{2k-1}-s_{2k}}} \nonumber \\
    &\qquad \times \sum_{t=-s_{2k}-1}^{s_{2k}}  \frac{q^{\frac{1}{2}t^2+t}}{(q;q)_{s_{2k}-t}(q;q)_{s_{2k}+t+1}} \nonumber \\
    &=q^{\frac{2k+3}{8}}\sum_{s_1\geq s_2\geq \cdots \geq s_{2k-1}\geq s_{2k}\geq 0} \frac{q^{\frac{1}{2}(s_1^2+s_2^2+\cdots+s_{2k-1}^2+s_{2k}^2)+(s_2+s_4+\cdots+s_{2k-2}+s_{2k})}}{(q;q)_{s_1-s_2}(q;q)_{s_2-s_3}\cdots (q;q)_{s_{2k-2}-s_{2k-1}}(q;q)_{s_{2k-1}-s_{2k}}}\nonumber \\
    &\qquad \times \frac{(-q^{\frac{3}{2}};q)_{s_{2k}}(-q^{-\frac{1}{2}};q)_{s_{2k}+1}}{(q;q)_{2s_{2k}+1}} \quad \text{(by \eqref{Andrews2})}\nonumber \\
    &=q^{\frac{2k-1}{8}}\sum_{s_1=0}^\infty q^{\frac{1}{2}s_1^2}(-q^{\frac{3}{2}};q)_{s_1} \sum_{s_2=0}^{s_1} \frac{q^{\frac{1}{2}s_2^2+s_2}(-q^{\frac{1}{2}};q)_{s_2}}{(q;q)_{s_1-s_2}(-q^{\frac{3}{2}};q)_{s_1}}  \cdots \nonumber \\
    &\qquad \times \sum_{s_{2k-1}=0}^{s_{2k-2}} \frac{q^{\frac{1}{2}s_{2k-1}^2}(-q^{\frac{3}{2}};q)_{s_{2k-1}}}{(q;q)_{s_{2k-2}-s_{2k-1}}(-q^{\frac{1}{2}};q)_{s_{2k-2}}}\sum_{s_{2k}=0}^{s_{2k-1}} \frac{q^{\frac{1}{2}s_{2k}^2+s_{2k}}(-q^{\frac{1}{2}};q)_{s_{2k}}}{(q;q)_{s_{2k-1}-s_{2k}}(-q^{\frac{3}{2}};q)_{s_{2k-1}}}\nonumber\\
    &\qquad \times\beta_{s_{2k}}^{(1)}(q;q).   
\end{align}
Here $(\alpha_n^{(1)}(q;q),\beta_n^{(1)}(q;q))$ is the Bailey pair \eqref{BP-new}. Applying \eqref{S3-aq} and \eqref{S4-aq} to this Bailey pair and iterating $i$ times ($i=1,2,\dots,k-1$), we obtain the Bailey pairs $(\alpha_n^{(1+2i)}(q;q),\beta_n^{(1+2i)}(q;q))$ with
\begin{align}
    \alpha_n^{(1+2i)}(q;q)=q^{i(n^2+n)}\alpha_n^{(1)}(q;q).
\end{align}
Applying \eqref{S3-aq} again we obtain the Bailey pair $(\alpha_n^{(2k)}(q;q),\beta_n^{(2k)}(q;q)$ with 
\begin{align}
\alpha_n^{(2k)}(q;q)=\frac{1+q^{\frac{1}{2}}}{1+q^{n+\frac{1}{2}}}q^{\frac{1}{2}n^2+n}\alpha_n^{(2k-1)}(q;q).
\end{align}
We have
\begin{align}\label{proof-thm2-odd-S1}
    S_1(q)&=q^{\frac{2k-1}{8}}\sum_{s_1=0}^\infty q^{\frac{1}{2}s_1^2}(-q^{\frac{3}{2}};q)_{s_1}\beta_{s_1}^{(2k)}(q;q) \nonumber \\
    &=q^{\frac{2k-1}{8}}\frac{(-q^{\frac{1}{2}};q)_\infty}{(q^2;q)_\infty} \sum_{n=0}^\infty \frac{1+q^{n+\frac{1}{2}}}{1+q^{\frac{1}{2}}}q^{\frac{1}{2}n^2}\alpha_n^{(2k)}(q;q) \quad\text{(by \eqref{S4-aq-infinite})}\nonumber \\
     &=q^{\frac{2k-1}{8}}\frac{(-q^{\frac{1}{2}};q)_\infty}{(q^2;q)_\infty} \sum_{n=0}^\infty q^{k(n^2+n)}\alpha_n^{(1)}(q;q) \nonumber \\
     &=q^{\frac{2k-1}{8}}\frac{(-q^{\frac{1}{2}};q)_\infty}{(q;q)_\infty}\Big(\sum_{m=0}^\infty (-1)^mq^{(4k+3)m^2+(2k+\frac{1}{2})m}(1+q^{2m+\frac{1}{2}})\nonumber \\
     &\qquad +\sum_{m=-\infty}^{-1} (-1)^mq^{(4k+3)m^2+(2k+\frac{1}{2})m}(1+q^{2m+\frac{1}{2}})\Big)\nonumber \\
     &=q^{\frac{2k-1}{8}}\frac{(-q^{\frac{1}{2}};q)_\infty}{(q;q)_\infty}\Big(\sum_{m=-\infty}^\infty (-1)^mq^{(4k+3)m^2+(2k+\frac{1}{2})m}\nonumber\\
     &\quad\quad\quad\quad\quad\quad\quad\quad+q^{\frac{1}{2}}\sum_{m=-\infty}^\infty (-1)^mq^{(4k+3)m^2+(2k+\frac{5}{2})m}\Big) \nonumber \\
     &=q^{\frac{2k-1}{8}}\frac{(-q^{\frac{1}{2}};q)_\infty}{(q;q)_\infty}\Big((q^{2k+\frac{5}{2}},q^{6k+\frac{7}{2}},q^{8k+6};q^{8k+6})_\infty\nonumber\\
     &\quad\quad\quad\quad\quad\quad\quad\quad+q^{\frac{1}{2}}(q^{2k+\frac{1}{2}},q^{6k+\frac{11}{2}},q^{8k+6};q^{8k+6})_\infty\Big).
\end{align}
Adding \eqref{proof-thm2-odd-S0} and \eqref{proof-thm2-odd-S1} together, we obtain \eqref{eq-thm2-odd}.
\end{proof}

\subsection*{Acknowledgements}
This work was supported by the National Key R\&D Program of China (Grant No.\ 2024YFA1014500).

\end{document}